\documentclass[12pt]{amsart}
\usepackage[colorlinks=true,citecolor=magenta,linkcolor=magenta]{hyperref}
\usepackage{amssymb,verbatim,amscd,amsmath,graphicx}
\usepackage{graphicx}
\usepackage{enumerate}
\usepackage{graphicx}
\usepackage{caption}
\usepackage{subcaption}
\usepackage{pdfsync}
\usepackage{color,colortbl}
\usepackage{fullpage}
\usepackage{bbm}
\usepackage{comment}
\usepackage{colonequals}
\numberwithin{equation}{section}
\newtheorem{theorem}{Theorem}[section]
\newtheorem{lemma}[theorem]{Lemma}

\newtheorem{corollary}[theorem]{Corollary}

\theoremstyle{definition}
\newtheorem{definition}[theorem]{Definition}
\newtheorem{conjecture}[theorem]{Conjecture}
\newtheorem{def-prop}[theorem]{Definition-Proposition}

\newtheorem{remark}[theorem]{Remark}
\newtheorem{example}[theorem]{Example}

\newtheorem*{acknowledgement}{Acknowledgements}

\newtheorem{question}[theorem]{Question}

\DeclareMathOperator{\reg}{reg}

\DeclareMathOperator{\Ass}{Ass}

\renewcommand{\AA}{{\mathbb A}}

\newcommand{\PP}{{\mathbb P}}

\newcommand{\NN}{{\mathbb N}}

\newcommand{\XX}{{\mathbb X}}

\newcommand{\VV}{{\mathbb V}}
\newcommand{\kk}{{\mathbbm k}}

\newcommand{\m}[1]{\marginpar{\addtolength{\baselineskip}{-3pt}{\footnotesize \it #1}}}

\def\F{{\mathcal F}}

\def\H{{\mathcal H}}

\def\mm{{\mathfrak m}}

\def\pp{{\frak p}}

\def\a{{\bf a}}

\def\m{{\bf m}}

\def\z{{\bf z}}

\def\ahat{\widehat{\alpha}}

\def\1{{\bf 1}}
\def\0{{\bf 0}}

\begin{document}
	
\title{Demailly's Conjecture and the Containment Problem}

\author{Sankhaneel Bisui}
\address{Tulane University \\ Department of Mathematics \\
	6823 St. Charles Ave. \\ New Orleans, LA 70118, USA}
\email{sbisui@tulane.edu}

\author{Elo\'isa Grifo}
\address{University of California at Riverside \\ Department of Mathematics \\
    900 University Ave. \\ Riverside, CA 92521, USA}
\email{eloisa.grifo@ucr.edu}

\author{Huy T\`ai H\`a}
\address{Tulane University \\ Department of Mathematics \\
	6823 St. Charles Ave. \\ New Orleans, LA 70118, USA}
\email{tha@tulane.edu}

\author{Th\'ai Th\`anh Nguy$\tilde{\text{\^e}}$n}
\address{Tulane University \\ Department of Mathematics \\
	6823 St. Charles Ave. \\ New Orleans, LA 70118, USA}
\email{tnguyen11@tulane.edu}


\begin{abstract}
We investigate Demailly's Conjecture for a general set of sufficiently many points. Demailly's Conjecture generalizes Chudnovsky's Conjecture in providing a lower bound for the Waldschmidt constant of a set of points in projective space. We also study a containment between symbolic and ordinary powers  conjectured by Harbourne and Huneke that in particular implies Demailly's bound, and prove that a general version of that containment holds for generic determinantal ideals and defining ideals of star configurations.
\end{abstract}

\maketitle


\section{Introduction} \label{sec.intro}

Let $\kk$ be a field, let $N \in \NN$ be an integer, let $R = \kk[\PP^N_\kk]$ be the homogeneous coordinate ring of $\PP^N_\kk$, and let $\mm$ be its maximal homogeneous ideal. For a homogeneous ideal $I \subseteq R$, let $\alpha(I)$ denote the least degree of a homogeneous polynomial in $I$, and let 
$$I^{(n)} \colonequals \bigcap_{\pp \in \Ass(R/I)} I^n R_\pp \cap R$$
denote its $n$-th \emph{symbolic power}. In studying the fundamental question of what the least degree of a homogeneous polynomial vanishing at a given set of points in $\PP^N_\kk$ with a prescribed order can be, Chudnovsky \cite{Chudnovsky1981} made the following conjecture.

\begin{conjecture}[Chudnovsky]
	\label{conj.Chud}
	Suppose that $\kk$ is an algebraically closed field of characteristic 0. Let $I$ be the defining ideal of a set of points $\XX\subseteq \PP^N_\kk$. Then, for all $n \geqslant 1$,
\begin{align}
\tag{C}
\dfrac{\alpha(I^{(n)})}{n} \geqslant \dfrac{\alpha(I)+N-1}{N}. \label{eq.Chudnovsky}
\end{align}
\end{conjecture}

Chudnovsky's Conjecture has been investigated extensively, for example in \cite{EsnaultViehweg, BoH, HaHu, GHM2013, Dumnicki2015, DTG2017, FMX2018, BGHN2020}. In particular, the conjecture was proved for a \emph{very general} set of points \cite{FMX2018} (\cite{DTG2017} also proved the conjecture in this case but for at least $2^N$ points) and for a \emph{general} set of sufficiently many points \cite{BGHN2020}. The conjecture was also generalized by Demailly \cite{Demailly1982} to the following statement.

\begin{conjecture}[Demailly] \label{conj.Demailly}
	Suppose that $\kk$ is an algebraically closed field of characteristic 0. Let $I$ be the defining ideal of a set of points $\XX \subseteq \PP^N_\kk$ and let $m \in \NN$ be any integer. Then, for all $n \geqslant 1$,
\begin{align}
\tag{D}
\dfrac{\alpha(I^{(n)})}{n} \geqslant \dfrac{\alpha(I^{(m)}) + N-1}{m+N-1}. \label{eq.Demailly}
\end{align}
\end{conjecture}

Demailly's Conjecture for $N = 2$ was proved by Esnault and Viehweg \cite{EsnaultViehweg}. Recent work of Malara, Szemberg and Szpond \cite{MSS2018}, extended by Chang and Jow \cite{CJ2020}, showed that for a fixed integer $m$, Demailly's Conjecture holds for a \emph{very general} set of sufficiently many points. Specifically, it was shown that, given $N \geqslant 3$, $m \in \NN$ and $s \geqslant (m+1)^N$, for each $n \geqslant 1$ there exists an open dense subset $U_n$ of the Hilbert scheme of $s$ points in $\PP^N_\kk$ such that Demailly's bound (\ref{eq.Demailly}) for $\alpha(I^{(n)})$ holds for $\XX \in U_n$. As a consequence, Demailly's Conjecture holds for all $\XX \in \bigcap_{n=1}^\infty U_n$. Chang and Jow \cite{CJ2020} further proved that if $s = k^N$, for some $k \in \NN$, then one can take $U_n$ to be the same for all $n \geqslant 1$, i.e., Demailly's Conjecture holds for a \emph{general} set of $k^N$ points. 

In this paper, we establish Demailly's Conjecture for a \emph{general} set of sufficiently many points. More precisely, we show that given $N \geqslant 3, m \in \NN$ and $s \geqslant (2m+3)^N$, there exists an open dense subset $U$ of the Hilbert scheme of $s$ points in $\PP^N_\kk$ such that Demailly's bound (\ref{eq.Demailly}) holds for $\XX \in U$ and all $n \geqslant 1$.

\medskip

\noindent \textbf{Theorem \ref{thm.main}.} Suppose that $\kk$ is algebraically closed (of arbitrary characteristic) and $N \geqslant 3$. For a fixed integer $m \geqslant 1$, let $I$ be the defining ideal of a general set of $s \geqslant (2m+3)^N$ points in $\PP^N_\kk$. For all $n \geqslant 1$, we have
$$\dfrac{\alpha(I^{(n)})}{n} \geqslant \dfrac{\alpha(I^{(m)}) + N-1}{m+N-1}.$$

\medskip

To prove Theorem \ref{thm.main}, we use a similar method to the one we used in our previous work \cite{BGHN2020}, where we proved Chudnovsky's Conjecture for a general set of sufficiently many points. This is not, however, a routine generalization. In \cite{BGHN2020}, Chudnovsky's bound (\ref{eq.Chudnovsky}) was obtained via the (Stable) Harbourne--Huneke Containment, which states that for a homogeneous radical ideal $I \subseteq R$ of big height $h$ we have
$$I^{(hr)} \subseteq \mm^{r(h-1)}I^r \ \text{ for } r \gg 0.$$
To achieve the Stable Harbourne--Huneke Containment, we showed that one particular containment $I^{(hc-h)} \subseteq \mm^{c(h-1)}I^c$, for some value $c \in \NN$, would lead to the stable containment $I^{(hr-h)} \subseteq \mm^{r(h-1)}I^r$ for $r \gg 0$. In a similar manner, Demailly's bound (\ref{eq.Demailly}) would follow as a consequence of the following more general version of the (Stable) Harbourne--Huneke Containment:
\begin{align}
\tag{HH}
	I^{(r(m+h-1))} \subseteq \mm^{r(h-1)}(I^{(m)})^r \ \text{ for } r \gg 0. \label{eq.SHH}
\end{align}
Unfortunately, this is where the generalization of the arguments in \cite{BGHN2020} breaks down. We cannot prove that one such containment would lead to the stable containment. To overcome this obstacle, we show that the stronger containment $I^{(c(m+h-1)-h+1)} \subseteq \mm^{c(h-1)}(I^{(m)})^c$, for some value $c \in \NN$, would imply $I^{(r(m+h-1))} \subseteq \mm^{r(h-1)}(I^{(m)})^r$ for \emph{infinitely many} values of $r$, and this turns out to be enough to obtain Demailly's bound.

It is an open problem whether, for a homogeneous radical ideal $I$, the general version of the Stable Harbourne--Huneke Containment stated in (\ref{eq.SHH}) holds; this problem is open even in the case where $I$ defines a set of points in $\PP^N_\kk$. In the second half of the paper, we investigate the general containment problem. We show that the containment holds for \emph{generic determinantal ideals} and the defining ideals of \emph{star configurations} in $\PP^N_\kk$. Our results are stated as follows.

\newpage

\noindent\textbf{Theorem \ref{thm.StarConf} and Remark \ref{rmk.HHgeneralDI}.} 
Let $\kk$ be a field.
\begin{enumerate}
	\item Let $I$ be the defining ideal of a codimension $h$ star configuration in $\PP^N_\kk$, for $h \leqslant N$. For any $m, r, c \geqslant 1$, we have
$$I^{(r(m+h-1)-h+c)} \subseteq \mm^{(r-1)(h-1)+c-1}(I^{(m)})^r.$$
	\item Let $I = I_t(X)$ be the ideal of $t$-minors of a matrix $X$ of indeterminates, and let $h$ denote its height in $\kk[X]$. For all $m,r,c \geqslant 1$, we have
$$I^{(r(m+h-1)-h+c)} \subseteq \mm^{(r-1)(h-1)+c-1}(I^{(m)})^r.$$
\end{enumerate}
In particular, if $I$ is the defining ideal of a star configuration or a generic determinantal ideal, then $I$ satisfies a Demailly-like bound, i.e., for all $n \geqslant 1$ we have
	$$\dfrac{\alpha(I^{(n)})}{n} \geqslant \frac{\alpha(I^{(m)}) + h - 1}{m + h -1}.$$
\medskip

Determinantal ideals are classical objects in both commutative algebra and algebraic geometry that have been studied extensively. The list of references is too large to be exhausted here; we refer the interested reader to \cite{determinantalrings} and references therein. In this paper, we are particularly interested in \emph{generic} determinantal ideals. Specifically, for a fixed pair of integers $p$ and $q$, let $X$ be a $p \times q$ matrix of indeterminates and let $R = \kk[X]$ be the corresponding polynomial ring. For $t \leqslant \min \{p,q\}$, let $I_t(X)$ be the ideal in $R$ generated by the $t$-minors of $X$; that is, $I_t(X)$ is generated by the determinants of all $t \times t$ submatrices of $X$. It is a well-known fact that $I_t(X)$ is a prime and Cohen-Macaulay ideal of height $h = (p-t+1)(q-t+1)$.

Star configurations have also been much studied in the literature with various applications \cite{CVT2011, CGVT2014, CGVT2015, Toh2015, PS2015, Toh2017, ACT2017,StarSteiner,WaldschmidtSpecialConfig}. They often provide good examples and a starting point in investigating algebraic invariants and properties of points in projective spaces; for instance, the minimal free resolution (cf. \cite{AS2012, RZ2016}), weak Lefschetz property (cf. \cite{Shin2012, AS2012, KS2016}), and symbolic powers and containment of powers (cf. \cite{GHM2013, HM2018, Shin2019, Mantero}). 

We shall use the most general definition of a star configuration given in \cite{Mantero}. Let $\F = \{F_1, \dots, F_n\}$ be a collection of homogeneous polynomials in $R$ and let $h < \min\{n,N\}$ be an integer. Suppose that any $(h+1)$ elements in $\F$ form a \emph{complete intersection}. The defining ideal of the \emph{codimension $h$ star configuration given by $\F$} is defined to be
$$I_{h,\F} = \bigcap_{1 \leqslant i_1 < \dots < i_h \leqslant n} (F_{i_1}, \dots, F_{i_h}).$$

To prove Theorems \ref{thm.StarConf} and \ref{determinantal case thm}, we use arguments similar to those in \cite{CEHH2017}, where the containment has been proved for squarefree monomial ideals. Note that, by a recent result of Mantero \cite{Mantero}, it is known that symbolic powers of the defining ideal of a star configuration $I_{h,\F}$ are generated by ``monomials'' in the elements of $\F$. A similar description for symbolic powers of determinantal ideals $I_t(X)$ was given by DeConcini, Eisenbud, and Procesi \cite{DEPdeterminantal}.

\begin{acknowledgement}
	The second author thanks Jack Jeffries for helpful discussions. The second author is supported by NSF grant DMS-2001445. The third author is partially supported by Louisiana Board of Regents, grant \# LEQSF(2017-19)-ENH-TR-25. The authors thank the anonymous referee for a careful read and many useful suggestions.
\end{acknowledgement}


\section{Demailly's Conjecture for general points} \label{sec.Demailly}

In this section, we establish Demailly's Conjecture for a general set of sufficiently many points. Recall first that for a homogeneous ideal $I \subseteq R$, the \emph{Waldschmidt constant} of $I$ is defined to be
$$\ahat(I) = \lim_{n \rightarrow \infty} \dfrac{\alpha(I^{(n)})}{n}.$$
It is known (cf. \cite[Lemma 2.3.1]{BoH}) that the Waldschmidt constant of an ideal exists and
$$\ahat(I) = \inf_{n \in \NN} \dfrac{\alpha(I^{(n)})}{n}.$$
Thus, Demailly's Conjecture can be equivalently stated as follows.

\begin{conjecture}[Demailly]
	Let $\kk$ be an algebraically closed field of characteristic 0. Let $I \subseteq \kk[\PP^N_\kk]$ be the defining ideal of a set of points in $\PP^N_\kk$ and let $m\in \NN$ be any integer. Then,
\begin{align}
\tag{D'}
\ahat(I) \geqslant \dfrac{\alpha(I^{(m)})+N-1}{m+N-1}. \label{eq.DemaillyNew}
\end{align}
\end{conjecture}

Demailly's Conjecture for $N = 2$ follows from \cite[Th\'eor\`eme I and Remarque (1.1)]{EsnaultViehweg}. Thus, for the remaining of the paper, we shall make the assumption that $N \geqslant 3$. We start by showing that Demailly's bound (\ref{eq.DemaillyNew}) follows from one appropriate containment between symbolic and ordinary powers of the given ideal. This result generalizes \cite[Proposition 5.3]{BGHN2020}.

\begin{lemma}
	\label{lem.14all}
	Let $I \subseteq R$ be an ideal of big height $h$ and let $m \in \NN$. Suppose that for some constant $c \in \NN$, we have $I^{(c(h+m-1)-h+1)} \subseteq \mm^{c(h-1)}\left(I^{(m)}\right)^c$. Then, 
	$$\ahat(I) \geqslant \dfrac{\alpha\left(I^{(m)}\right)+h-1}{m+h-1}.$$
\end{lemma}

\begin{proof} 
We will make use of a result of Ein--Lazarsfeld--Smith \cite[Theorem 2.2]{ELS} and Hochster--Huneke \cite[Theorem 1.1 (a)]{comparison}, which says that $I^{(ht+kt)} \subseteq \left( I^{(k+1)} \right)^t$ for all $t \geqslant 1$ and all $k \geqslant 0$. We obtain that for all $t \in \NN$,
	\begin{align*}
	I^{(ct(m+h-1))} & = I^{(ht+t[c(m+h-1)-h])} \\
	& \subseteq \left(I^{(c(m+h-1)-h+1)}\right)^t \\
	& \subseteq \left[\mm^{c(h-1)}(I^{(m)})^c\right]^t \\
	& =  \mm^{ct(h-1)} \left(I^{(m)}\right)^{ct}.
	\end{align*}
In particular, it follows that
$$\dfrac{\alpha\left(I^{(ct(m+h-1))}\right)}{ct(m+h-1)} \geqslant \dfrac{ct(h-1) + ct\alpha\left(I^{(m)}\right)}{ct(m+h-1)} = \dfrac{\alpha(I^{(m)}) + h-1}{m+h-1}.$$
By taking the limit as $t \rightarrow \infty$, it follows that
$$\ahat(I) \geqslant \dfrac{\alpha(I^{(m)})+h-1}{m+h-1}.$$
The assertion is proved.
\end{proof}

In light of Lemma \ref{lem.14all}, to prove Demailly's Conjecture for the defining ideal $I$ of a set of points in $\PP^N_\kk$, the task at hand is to exhibit the containment $I^{(c(h+m-1)-h+1)} \subseteq \mm^{c(h-1)}\left(I^{(m)}\right)^c$ for a specific constant $c$. Our method is to use \emph{specialization} techniques, in a similar manner to what we have done in \cite{BGHN2020}, to reduce the problem to the \emph{generic} set of points in $\PP^N_{\kk(\z)}$. 

We shall now recall the definition of specialization in the sense of Krull \cite{Krull1948}. Let $\z = (z_{ij})_{1 \leqslant i \leqslant s, 0 \leqslant j \leqslant N}$ be the collection of $s(N+1)$ new indeterminates. Let
$$P_i(\z) = [z_{i0}: \dots : z_{iN}] \in \PP^N_{\kk(\z)} \quad \text{ and } \quad \XX(\z) = \{P_1(\z), \dots, P_s(\z)\}.$$
The set $\XX(\z)$ is the set of $s$ \emph{generic} points in $\PP^N_{\kk(\z)}$. Given $\a = (a_{ij})_{1 \leqslant i \leqslant s, 0 \leqslant j \leqslant N} \in \AA^{s(N+1)}_\kk$, let $P_i(\a)$ and $\XX(\a)$ be obtained from $P_i(\z)$ and $\XX(\z)$ by setting $z_{ij} = a_{ij}$ for all $i,j$. It is easy to see that there exists an open dense subset $W_0 \subseteq \AA^{s(N+1)}_\kk$ such that $\XX(\a)$ is a set of distinct points in $\PP^N_\kk$ for all $\a \in W_0$ (and all subsets of $s$ points in $\PP^N_\kk$ arise in this way). 

The following result allows us to focus on open dense subsets of $\AA^{s(N+1)}_\kk$ when discussing general sets of points in $\PP^N_\kk$.

\begin{lemma}[\protect{\cite[Lemma 2.3]{FMX2018}}] \label{lem.FMX}
	Let $W \subseteq \AA^{s(N+1)}_\kk$ be an open dense subset such that a property $\mathcal{P}$ holds for $\XX(\a)$ whenever $\a \in W$. Then, the property $\mathcal{P}$ holds for a general set of $s$ points in $\PP^N_\kk$.
\end{lemma}

To get the desired containment for the generic set of points in $\PP^N_{\kk(\z)}$ we shall need the following combinatorial lemma, which is a generalization of \cite[Lemma 4.6]{BGHN2020} and \cite[Lemma 3.1]{MSS2018}.

\begin{lemma}\label{demaillyApproxi} 
	Suppose that $N \geqslant 3$ and $k \geqslant 2m+2$. We have
	$${{(k-1)(m+N-1)+N-1} \choose N } \geqslant {(k+1)^N{m+N-1 \choose N} }.$$	
\end{lemma}	

\begin{proof} We shall use induction on $N$. For $N=3$, we need to show that 
	$$	{{(k-1)(m+2)+2} \choose 3 } \geqslant {(k+1)^3{m+2 \choose 3} },$$ 
	which is equivalent to 
	$$(k-1)[(k-1)(m+2)+2][(k-1)(m+2)+1]\geqslant (k+1)^3(m+1)m.$$
Set $k' = k-1$. It follows that $k' \geqslant 2m+1$, and we need to prove the inequality 
	\begin{align*}
& 	k'[k'(m+2)+2][k'(m+2)+1] \geqslant (k'+2)^3(m+1)m, \text{ i.e.,} \\
&	(3m+4)k'^3-3(2m^2+m-2)k'^2-(12m^2+12m-2)k'-8(m^2+m)  \geqslant 0.
	\end{align*}
By setting $f(k')$ to be the left hand side of this inequality, as a function in $k'$, it suffices to show that $f(k')$ is an increasing function for $k' \geqslant 2m+1$ and $f(2m+1) \geqslant 0$.
	
It is easy to see that $f(2m+1)=4(m+1)^2(2m+3) >0$. On the other hand, we have
	$$f'(k')= 3(3m+4)k'^2-6(2m^2+m-2)k'-(12m^2+12m-1).$$ 
	We will show that $2m+1$ is greater than both roots of $f'(k')$. Indeed, the bigger root of $f'(k')$ is 
	$$k'_1=\dfrac{3(2m^2+m-2)+\sqrt{3}\sqrt{12m^4+48m^3+63m^2+33m+8}}{3(3m+4)}.$$
	Since 
\begin{align*}
	[(2m+1)3(3m+4)-6m^2-3m+6]^2-3(12m^4+48m^3+63m^2+33m+8) \\
	=3(36m^4+192m^3+381m^2+327m+100) & >0,
\end{align*}
we have $2m+1>k_1$. 
	This establishes the desired inequality for $N = 3$.
	
Suppose now that the desired inequality holds for $N \geqslant 3$, i.e.,
	$${{(k-1)(m+N-1)+N-1} \choose N } \geqslant {(k+1)^N{m+N-1 \choose N} }.$$	
We shall prove that the inequality holds for $N+1$ as well. That is, 
	$${{(k-1)(m+N)+N} \choose N+1 } \geqslant {(k+1)^{N+1}{m+N \choose N+1} }.$$	
	
Set $x = (k-1)(m+N-1)+N-1$. Then $x+k=(k-1)(m+N)+N$, and we need to prove that, for $k \geqslant 2m+2$, 
\begin{align}
{x + k \choose N+1 }  \geqslant (k+1)^{N+1}{m+N \choose N+1}. \label{eq.combinatorial1}
\end{align}
Indeed, by the induction hypothesis, we have
	\begin{align*}
	{x + k \choose N+1 } & = {x \choose N} \,\frac{(x+k)\dots (x+1)}{(N+1)(x-N+1)\dots (x-N+k-1)}\\
	& \geqslant (k+1)^{N+1}{m+N\choose N+1} \frac{(N+1)}{(k+1)(m+N)} \cdot \frac{(x+k)\dots (x+1)}{(N+1)(x-N+1)\dots (x-N+k-1)}.
	\end{align*}	
	Hence, it is enough to show that if $k \geqslant 2m+2$ then
\begin{align}
(x+k)(x+k-1)\dots (x+1) \geqslant (k+1)(m+N)(x-N+1)\dots (x-N+k-1). \label{eq.combinatorial2}
\end{align}

Observe that $x+i \geqslant x-N+i+1$. Thus, to prove (\ref{eq.combinatorial2}), it suffices to show that $(x+k)(x+k-1) \geqslant (k+1)(m+N) (x-N+1)$. That is,
	$$[(k-1)(m+N)+N][(k-1)(m+N)+N-1] \geqslant (k+1)(m+N)[(k-1)(m+N)-(k-1)].$$
This inequality, by setting $k' = k-1$, is equivalent to
	$$(m+N)[k'^2-(2m-1)k']+N(N-1) \geqslant 0.$$
The last inequality clearly holds for $k' \geqslant 2m+1$. Hence, (\ref{eq.combinatorial1}) and (\ref{eq.combinatorial2}) hold for $k \geqslant 2m+2$. This completes the proof.
\end{proof}	

\begin{remark} 
	\label{rmk.largerN}
	For $N \geqslant 4$, we can slightly improve the bound for $k$ in Lemma \ref{demaillyApproxi} to be $k \geqslant 2m+1$ or $k \geqslant 2m$ if, in addition, $m \geqslant 3$. For $N\geqslant 5$ and $m\geqslant 2$, the bound is also $k\geqslant 2m$.
\end{remark}

We shall also make use of the following result of Trung and Valla \cite[Theorem 2.4]{trung1995upper}, which we restate the special case when $m_1 = \dots = m_s = m$ for our purpose.

\begin{lemma}[\protect{See \cite[Theorem 2.4]{trung1995upper}}] \label{lem.TV95}
	Assume that $\kk$ is algebraically closed. Let $X$ be a general set of $s$ points in $\PP^N_\kk$ and let $I$ be its defining ideal. Let $m \in \NN$ and let $w$ denote the least integer such that 
	$$(s-1){m+N-1 \choose N} < {N+w \choose N}.$$
	Then, we have
	$\reg(I^{(m)}) \leqslant m+w.$
\end{lemma}

Note that the regularity \emph{specializes}, following for instance \cite[Theorem 4.2]{NhiTrung1999}. Thus, Lemma \ref{lem.TV95} holds for $X = \XX(\z)$, the set of $s$ generic points in $\PP^N_{\kk(\z)}$.

In the next few lemmas, we establish a general version of the Stable Harbourne--Huneke Containment for the defining ideal of sufficiently many generic points in $\PP^N_{\kk(\z)}$.

\begin{lemma}[\protect{Compare with \cite[Lemma 4.7]{BGHN2020}}] 
	\label{generalHaHu}
	Let $\kk$ be an algebraically closed field. Suppose that $s \geqslant (2m+2)^N$ and $N \geqslant 3$. Let $I(\z)$ be the defining ideal of $s$ generic points in $\PP^N_{\kk(\z)}$. For $r \gg 0$, we have
	$$I(\z)^{(r(m+N-1)-N+1)}  \subseteq (I(\z)^{(m)})^r.$$
\end{lemma}	

\begin{proof} 
For simplicity of notation, we shall write $I$ for $I(\z)$ in this lemma. Let $k \geqslant 2m+2$ be the integer such that $k^N \leqslant s < (k+1)^N$. It follows from \cite[Theorem 2]{DTG2017} that $\ahat(I) \geqslant \lfloor \sqrt[N]{s}\rfloor = k$. As in \cite[Lemma 4.7]{BGHN2020}, we remark here that even though \cite[Theorem 2]{DTG2017} assumed $\text{char } \kk = 0$, its proof carries through for any infinite field $\kk$. In particular, we have
	$$\alpha(I^{(r(m+N-1)-N+1) }) \geqslant k[r(m+N-1)-N+1 ].$$
	
	By Lemmas \ref{demaillyApproxi} and \ref{lem.TV95}, for $r \gg 0$, we have 
	\begin{align*}
	\reg(I^{(m)}) & \leqslant m +[(k-1)(m+N-1) -1]\\
	& \leqslant m +(k-1)(m+N-1) -\frac{k}{r}(N-1)\\
	& \leqslant k(m+N-1) -\frac{k}{r}(N-1).
	\end{align*}	
This implies that $r \reg(I^{(m)}) \leqslant \alpha(I^{(r(m+N-1)-N+1)})$. Thus, we obtain the following inequality for the saturation degree of $(I^{(m)})^r$:
	$$\textrm{sat}( (I^{(m)})^r)  \leqslant r \reg(I^{(m)}) \leqslant \alpha(I^{(r(m+N-1)-N+1) }).$$
As a consequence, it follows that for $t \geqslant \alpha(I^{(r(m+N-1)-N+1)})$, 
$$\left[(I^{(m)})^{(r)}\right]_t = [(I^{(m)})^r]_t.$$
Since $r(m+N-1)-N+1 \geqslant mr$, we have $I^{(r(m+N-1)-N+1) }  \subseteq I^{(mr)} = \left( I^{(m)} \right)^{(r)}$, where the equality follows from the definition of symbolic powers. Hence, we conclude that $I^{(r(m+N-1)-N+1) }  \subseteq (I^{(m)})^{r}$.
\end{proof}

\begin{lemma}\label{demaillyContainment}
Suppose that $\kk$ is algebraically closed, $s \geqslant (2m+2)^N$ and $N \geqslant 3$. Let $I(\z)$ be the defining ideal of $s$ generic points in $\PP^N_{\kk(\z)}$. Let $\mm_\z$ denote the maximal homogeneous ideal in $\kk\big[\PP^N_{\kk(\z)}\big]$. For $r \gg 0$, we have
	$$I(\z)^{(r(m+N-1)-N+1)} \subseteq \mm_\z^{r(N-1)} (I(\z)^{(m)})^r.$$	
\end{lemma}	

\begin{proof} For simplicity of notation, we shall write $I$ for $I(\z)$ and $\mm$ for $\mm_\z$ in this lemma. Let $k \geqslant 2m+2$ be the integer such that $k^N \leqslant s < (k+1)^N$.
	By Lemma \ref{generalHaHu}, we have $I^{(r(m+N-1)-N+1)} \subseteq (I^{(m)})^r$ for $r \gg 0$. Thus, it suffices to show that, for $r \gg 0$, 
\begin{align}
\alpha(I^{(r(m+N-1)-N+1) }) \geqslant r \reg (I^{m}) + r (N-1).\label{eq.regularity}
\end{align}
	
As before, it follows from \cite[Theorem 2.4]{trung1995upper} and Lemma \ref{demaillyApproxi} that, for $r \gg 0$, 
	\begin{align*}
	\reg(I^{(m)}) & \leqslant m +[(k-1)(m+N-1) -1] - 1 + 1\\
	& \leqslant m +(k-1)(m+N-1) -\frac{k}{r}(N-1).
	\end{align*}
That is,
$$\reg(I^{(m)}) +N-1 \leqslant k(m+N-1) -  \frac{k}{r}(N-1).$$
Thus, for $r \gg 0$, we have
	$$r(\reg(I^{(m)}) +N-1 ) \leqslant rk(m+N-1) -k(N-1).$$                     
Furthermore, again by \cite[Theorem 2]{DTG2017}, we have $\ahat(I) \geqslant \lfloor \sqrt[N]{s} \rfloor = k$. In particular, it follows that
	$$\alpha(I^{(r(m+N-1)-N+1) }) \geqslant k[r(m+N-1)-N+1] = rk(m+N-1) - k(N-1).$$
Hence, (\ref{eq.regularity}) holds for $r \gg 0$, and the lemma is proved.
\end{proof}	

We are now ready to state our first main result, which establishes Demailly's Conjecture for a general set of sufficiently many points in $\PP^N_\kk$ for any algebraically closed field $\kk$ of arbitrary characteristic.

\begin{theorem}
	\label{thm.main}
	Suppose that $\kk$ is algebraically closed and $N \geqslant 3$. For a fixed integer $m \geqslant 1$, let $I$ be the defining ideal of a general set of $s \geqslant (2m+2)^N$ points in $\PP^N_\kk$. Then,
	$$\ahat(I) \geqslant \dfrac{\alpha(I^{(m)})+N-1}{m+N-1}.$$
\end{theorem}

\begin{proof} Let $I(\z)$ be the defining ideal of $s$ generic points in $\PP^N_{\kk(\z)}$ and let $\mm_\z$ denote the maximal homogeneous ideal of $\kk\big[\PP^N_{\kk(\z)}\big]$. It follows from Lemma \ref{demaillyContainment} that there exists a constant $c \in \NN$ such that 
	$$I(\z)^{(c(m+N-1)-N+1)} \subseteq \mm_\z^{c(N-1)} (I(\z)^{(m)})^c.$$
This, together with \cite[Satz 2 and 3]{Krull1948} (see also \cite[Remark 2.10]{BGHN2020}), implies that there exists an open dense subset $U \subseteq \AA^{s(N+1)}$ such that for all $\a \in U$, we have
$$I(\a)^{(c(m+N-1)-N+1)} \subseteq \mm^{c(N-1)} (I(\a)^{(m)})^c.$$
The theorem now follows from Lemmas \ref{lem.14all} and \ref{lem.FMX}.
\end{proof}

\begin{remark} \label{rmk.sLargeN}
	By Remark \ref{rmk.largerN}, the bound for $s$ in Theorem \ref{thm.main} can be improved slightly when $N \geqslant 4$ to require only that $s \geqslant (2m+1)^N$ or $s \geqslant (2m)^N$ if, in addition, $m \geqslant 3$. For $N\geqslant 5$ and $m\geqslant 2$ it also requires only $s \geqslant (2m)^N$.
\end{remark}

\begin{remark}
	When $m=1$, Demailly's Inequality (\ref{eq.DemaillyNew}) coincides with Chudnovky's Conjecture, which we previously showed to hold for sufficiently many general points in $\mathbb{P}^N$ in \cite{BGHN2020}. Theorem \ref{thm.main} is a generalization of \cite[Theorem 5.1]{BGHN2020}, extending Chudnovsky's Conjecture for \emph{sufficiently many} general points to Demailly's Conjecture. In particular, \cite[Theorem 5.1]{BGHN2020} states that Chudnovsky's Conjecture holds for $s \geqslant 4^N$ general points in $\mathbb{P}^N_\kk$, which is the result in Theorem \ref{thm.main} when $m=1$. On the other hand, for $N \geqslant 4$, Remark \ref{rmk.sLargeN} shows our bound for $s \geqslant (2m+1)^N$ in Demailly's Conjecture, when $m = 1$, agrees with the bound $s \geqslant 3^N$ given in \cite[Remark 5.2]{BGHN2020} for Chudnovsky's Conjecture.
	
	Another crucial difference is that in \cite{BGHN2020} we also showed that the containment $I^{(rN)} \subseteq \mm^{r(N-1)} I^r$ holds for $r \gg 0$. Here the corresponding generalization would be $I^{(r(N+m-1))} \subseteq \mm^{r(N-1)} \left( I^{(m)} \right)^r$ for all $r \gg 0$ and all $m \geqslant 1$. Unfortunately, we have not been able to prove this stable containment for all $r$ sufficiently large; we only show it for infinitely many values of $r$.
\end{remark}
	

\section{Harbourne--Huneke containment beyond points}

In this section, we investigate a general containment between symbolic and ordinary powers of radical ideals, and show that this containment holds for generic determinantal ideals and the defining ideals of star configurations. Specifically, we are interested in the following general version of the Harbourne--Huneke Containment for radical ideals.

\begin{question}\label{question demailly containment general h}
	Let $I$ be either a radical ideal of big height $h$ in a regular local ring $(R,\mm)$, or a homogeneous radical ideal of big height $h$ in a polynomial ring $R$ with maximal homogeneous ideal $\mm$. Does the containment
	$$I^{(r(h+m-1))} \subseteq \mm^{r(h-1)} \left( I^{(m)}\right)^r$$
	hold for all $m,r \geqslant 1$?
\end{question}

A positive answer to Question \ref{question demailly containment general h} would in particular imply a Demailly-like bound for homogeneous radical ideals in $\kk[\mathbb{P}^N]$, i.e., an affirmative answer to the following question.

\begin{question}[Demailly-like bound]\label{question demailly-like bound}
	Let $R$ be a polynomial ring over $\kk$ and let $I$ be a homogeneous radical ideal of big height $h$ in $R$. Does the inequality 	
	$$\dfrac{\alpha(I^{(n)})}{n} \geqslant \frac{\alpha(I^{(m)}) + h - 1}{m + h -1}$$
	hold for all $n, m \geqslant 1$?
\end{question}

Question \ref{question demailly containment general h} was first asked for ideals of points by Harbourne and Huneke in \cite[Question 4.2.3]{HaHu}, and the more general version for radical ideals of big height $h$ appeared in \cite[Conjecture 2.9]{CEHH2017}. The answer to both questions is yes for squarefree monomial ideals by \cite[Corollary 4.3]{CEHH2017}, where in fact a stronger containment was established \cite[Theorem 4.2]{CEHH2017}. Similar containment for the defining ideal of a general set of points in $\PP^2_\kk$ were investigated in \cite{BCH2014}. Furthermore, by the same reasoning as in the previous section, the left hand side of the inequality in Question \ref{question demailly-like bound} can be replaced by the Waldschmidt constant $\ahat(I)$ of $I$. We refer the interested reader to \cite{CHHVT2020} for more information about the Waldschmidt constant, containment and equality between symbolic and ordinary powers of ideals.

Our goal in this section is show that Question \ref{question demailly containment general h} has a positive answer for special classes of ideals. In a natural approach to Question \ref{question demailly containment general h}, one might hope to make use of the following general containment of \cite{ELS, comparison}:
$$I^{(r(h+m-1))} \subseteq \left( I^{(m)}\right)^r.$$
Given this containment, to derive an affirmative answer to Question \ref{question demailly containment general h}, one could aim to simply show that 
$$\alpha \left( I^{(r(h+m-1))} \right) \geqslant r(h-1) + r \, \omega \left( I^{(m)}\right),$$
where, for a homogeneous ideal $J$, $\omega(J)$ denotes the maximal generating degree of $J$.
This inequality, however, is often false, as illustrated in the following examples.

\begin{example}\label{example fermat}
	Consider $n \geqslant 3$, a field $\kk$ with $\textrm{char } \kk \neq 2$ containing $n$ distinct roots of unity, and $R = \kk[x,y,z]$. The symbolic powers of the ideal
	$$I = \left( x(y^n-z^n), y(z^n-x^n), z(x^n-y^n) \right)$$
	have an interesting behavior; in particular, $I^{(3)} \nsubseteq I^2$ \cite{counterexamples, HaSeFermat}, and in fact the case $\kk = \mathbb{C}$ and $n=3$ was the first example ever found of an ideal of big height $2$ with such behavior \cite{counterexamples}. 
	
	By the proof of \cite[Theorem 2.1]{DHNSST2015}, $\alpha(I^{(3k)}) = 3nk$;
	moreover, by \cite[Theorem 3.6]{NagelSeceleanu}, $\omega(I^{(kn)}) = k(n+1)$ for all $k \geqslant 1$. Therefore, we immediately see that 
	$$\alpha \left( I^{(3(kn+1))} \right) = 3(kn+1)n \not\geqslant 3 + 3 kn (n+1) = 3 + 3 \, \omega(I^{(kn)}).$$ 
	In fact, Macaulay2 \cite{M2} computations with $n = 3$ suggest that
	$$\alpha \left( I^{(r(m+1))} \right) \geqslant r + r \, \omega \left( I^{(m)}\right)$$
	may never hold. However, this does not prevent the containment in Question \ref{question demailly containment general h},
	$$I^{(r(m+1))} \subseteq \mm^r \left( I^{(m)} \right)^r,$$
	to hold --- indeed Macaulay2 \cite{M2} computations support this containment for small values of $r$ and $m$ when $n=3$.
	If, in addition, $\text{char } \kk = 0$, then this containment holds for infinitely many values of $m$. Indeed, we have 
	$$I^{(r(mn+1))} = I^{(rmn+r)} \subseteq \mm^r I^{(rmn)}  = \mm^r \left( I^{(n)} \right)^{mr} = \mm^r \left( I^{(mn)} \right)^r.$$
	
	Note that Demailly's bound can be checked in this case, at least for multiples of $3$. Indeed, by the proof of \cite[Theorem 2.1]{DHNSST2015}, $\widehat\alpha(I) = n$ and $\alpha(I^{(3m)}) = 3nm$, so $I$ satisfies Demailly's bound for all multiples of $3$:
	$$\widehat\alpha(I) \geqslant \frac{\alpha(I^{(3m)}) + 1}{3m + 1}.$$
	
Furthermore, if $\text{char } \kk = 0$, then Demailly's bound can also be verified by taking powers of the form $3m+2$. In particular, since  $I^{(3m+3)} \subseteq \mm I^{(3m+2)}$, we have $\alpha(I^{(3m+3)}) \geqslant \alpha(I^{(3m+2)})+1$. Equivalently, we get $(3m+3)n \geqslant \alpha(I^{(3m+2)})+1$, or 
$$\widehat\alpha(I) \geqslant \frac{\alpha(I^{(3m+2)}) + 1}{3m+ 3}.$$
\end{example}

\begin{example}[Generic determinantal ideals]\label{example determinantal}
	Fix some $t \leqslant q \leqslant p$, let $X$ be an $p \times q$ matrix of indeterminates, and let $R = \kk[X]$ be the corresponding polynomial ring over a field $\kk$. Consider the ideal $I = I_t(X)$ of $t$-minors of $X$, which is a prime in $R$ of height $h = (p-t+1)(q-t+1)$. By \cite{ELS,comparison}, 
	$$I^{(r(h+m-1))} \subseteq \left( I^{(m)}\right)^r$$
	for all $m,r \geqslant 1$. To show that $I$ satisfies the containment proposed in Question \ref{question demailly containment general h}, one might attempt to check that for all $m, r \geqslant 1$,
	$$\alpha \left( I^{(r(h+m-1))} \right) \geqslant \omega \left( \mm^{r(h-1)} \left( I^{(m)}\right)^r \right) = r(h-1) + r \omega(I^{(m)}).$$
	However, this inequality does not always hold; for example, if $I$ is the ideal of $2 \times 2$ minors of a generic $3 \times 3$ matrix (meaning $p=q=3$ and $t =2$, so $h = 4$) and we take $r = 1$, $m=5$, it turns out that
	$$\alpha \left( I^{(r(h+m-1))} \right) = \alpha \left( I^{(8)} \right) = 12 < 13 = 3 + \omega(I^{(5)}) = r(h-1) + r \omega(I^{(m)}).$$
	Nevertheless, as we will show in Theorem \ref{determinantal case thm} that the containment in Question \ref{question demailly containment general h} holds for $I$, and as a consequence so does the inequality in Question \ref{question demailly-like bound}. 
\end{example}

Examples \ref{example fermat} and \ref{example determinantal} demonstrate that the obvious approach to establish the containment in Question \ref{question demailly containment general h} may not work. However, in the remaining of the paper, we shall see that this containment indeed holds for generic determinantal ideals and defining ideals of star configurations.


\subsection{Star configurations} \label{sec.Star}

We start by recalling the construction of a star configuration of hypersurfaces in $\PP^N_\kk$, following \cite{Mantero} (see also \cite{GHM2013}).

\begin{definition}
	Let $\H = \{H_1, \dots, H_n\}$ be a collection of $s \geqslant 1$ hypersurfaces in $\PP^N_\kk$. Assume that these hypersurfaces \emph{meet properly}; that is, the intersection of any $k$ of these hypersurfaces either is empty or has codimension $k$. For $1 \leqslant h \leqslant \min\{n,N\}$, let $\VV_{h,\H}$ be the union of the codimension $h$ subvarieties of $\PP^N_\kk$ defined by all the intersections of $h$ of these hypersurfaces. That is,
	$$\VV_{h,\H} = \bigcup_{1 \leqslant i_1 < \dots < i_h \leqslant n} H_{i_1} \cap \dots \cap H_{i_h}.$$
We call $\VV_{h,\H}$ a \emph{codimension $h$ star configuration}.
\end{definition}

Suppose that for $i = 1, \dots, n$, the hypersurface $H_i$ is defined by the homogeneous polynomial $F_i$. Set $\F = \{F_1, \dots, F_n\}$. Then, the defining ideal of $\VV_{h,\H}$ is given by
$$I_{h,\F} = \bigcap_{1 \leqslant i_1 < \dots < i_h \leqslant n} (F_{i_1}, \dots, F_{i_h}).$$
We refer to  $I_{h,\F}$ as the defining ideal of a codimension $h$ star configuration in $\PP^N_\kk$. Note that, since $h \leqslant N$, it further follows from the definition that any $(h+1)$ elements in $\F$ form a \emph{complete intersection}. A recent result of Mantero \cite[Theorem 4.9]{Mantero} shows that if, in addition, $h < n$, then $I_{h,\F}$ is minimally generated by appropriate \emph{monomials} in the homogeneous forms of $\F$.

\begin{theorem} \label{thm.StarConf}
Let $\kk$ be a field. Let $I$ be the defining ideal of a codimension $h$ star configuration in $\PP^N_\kk$, for some $h \leqslant N$. For any $m, r, c \geqslant 1$, we have
$$I^{(r(m+h-1)-h+c)} \subseteq \mm^{(r-1)(h-1)+c-1}(I^{(m)})^r.$$
\end{theorem}

\begin{proof} 
For a complete intersection, symbolic and ordinary powers are equal. Thus, the statement is trivial if $h = n$. Assume that $h < n$. Let $\F = \{F_1, \dots, F_n\}$ be the collection of homogeneous forms which defines the given star configuration. By definition, we have
\begin{align} 
I = \bigcap_{1 \leqslant i_1 < i_2 < \dots < i_h \leqslant n} (F_{i_1}, \dots, F_{i_h}). \label{eq.primaryD}
\end{align}
We shall proceed following a similar argument to the one in \cite[Theorem 4.2]{CEHH2017}, where the containment was proved for squarefree monomial ideals. 

Denote the ideals of the form $(F_{i_1}, \dots, F_{i_h})$ in (\ref{eq.primaryD}), in some order, by $Q_1,...,Q_s$, where $s= {n \choose h}$ and each $Q_i$ is of the form $Q_i = (F_{i_1}, \dots, F_{i_h})$. 
Let $E_i$ denote the set of indices of the elements from $\F$ appearing in $Q_i$, namely $E_i=\{i_1,...,i_h\}$. It follows from \cite[Proposition 2.4]{Mantero} (see also \cite{GHM2013}) that
\begin{align}
I^{(r(m+h-1)-h+c)}=\bigcap_{i=1}^s Q_i^{r(m+h-1)-h+c}. \label{eq.symbolicSC}
\end{align}
By \cite[Theorem 4.9]{Mantero}, $I^{(r(m+h-1)-h+1)}$ is generated by \emph{monomials} in the forms of $\mathcal{F}$. Consider an arbitrary such monomial $M = F_1^{a_1} \dots F_s^{a_s}$ in $I^{(r(m+h-1)-h+c)}$. 

Observe that if $j \not\in E_i = \{i_1, \dots, i_h\}$ then $\{F_j, F_{i_1}, \dots, F_{i_h}\}$ form a complete intersection. This implies that $F_j$ is not a zero-divisor of $Q_i^k$ for all $k \geqslant 1$. Thus, it follows from (\ref{eq.symbolicSC}) that, for each $i = 1, \dots, s$, we have
	$$\sum_{j\in E_i}a_j \geqslant r(m+h-1)-h+c.$$ 
Let $d_j \in \mathbb{Z}_{\geqslant 0}$ be such that $d_jr \leqslant a_j < (d_j+1)r$ and set $a_j' =a_j-d_jr \leqslant r-1$ (in particular, $\sum_{j \in E_i} a_j'  \leqslant (r-1)h$). It can be seen that
	$$\sum_{j \in E_i} d_jr = \sum_{j \in E_i} (a_j - a_j') \geqslant r(m+h-1) -h+c - (r-1)h = r(m-1)+c.$$ 
The left hand side of this inequality is divisible by $r$, so we deduce that
	$\sum_{j \in E_i} d_jr \geqslant rm.$
In particular, we have $\sum_{j\in E_i} d_j \geqslant m$.

Consider the system of inequalities $\left\{ \sum_{j\in E _i } d_j \geqslant m ~\big|~ i=1,\dots, s\right\}.$
	By successively reducing the values of $d_j$'s  we can choose 
	$ 0 \leqslant d_j' \leqslant d_j$ such that the system of inequalities 
	$$\left\{ \sum_{j\in E _i } d_j ' \geqslant m ~\Big|~ i=1,\dots, s\right\}$$
is still satisfied, but for at least one value of $1 \leqslant \ell \leqslant s$ we obtain the equality $\sum_{j\in E_\ell} d_j'=m$.

Set $ f= \prod_{j=1}^n F_j^{d_j'r} \text{ and } g = \prod_{j=1}^n F_j^{a_j-d_j'r}.$ Then, $M = fg$. Also, it follows from (\ref{eq.symbolicSC}) that $\prod_{j=1}^n F_j^{d_j'} \in I^{(m)}$. Thus, $f \in (I^{(m)})^r$.

On the other hand, it is easy to see that
\begin{align*}
\deg g & \geqslant \sum_{j \in E_\ell} (a_j - d_j'r)  \\
& = \sum_{j \in E_\ell} a_j - (\sum_{j \in E_\ell} d_j')r \\
& \geqslant r(m+h-1)-h+c - rm  \\
& = (r-1)(h-1)+c-1.
\end{align*}
Therefore, $g \in \mm^{(r-1)(h-1)+c-1}$. Hence, $M \in \mm^{(r-1)(h-1)+c-1}(I^{(m)})^r$, and the result follows.
\end{proof}

As a corollary of Theorem \ref{thm.StarConf}, we can show that Demailly's bound holds for a star configuration in $\PP^N_\kk$. 

\begin{corollary}
	\label{cor.DemaillySC}
	Let $I$ be the defining ideal of a codimension $h$ star configuration in $\PP^N_\kk$, for some $h \leqslant N$. For any $m, r \in \NN$, we have
	$$I^{(r(m+h-1))} \subseteq \mm^{r(h-1)}(I^{(m)})^r.$$
	In particular, Demailly-like bound holds for defining ideals of star configurations in $\PP^N_\kk$.
\end{corollary}

\begin{proof} 
The first statement is a consequence of Theorem \ref{thm.StarConf} by setting $c = h$. The second statement follows immediately from the containment $I^{(r(m+h-1))} \subseteq \mm^{r(h-1)}(I^{(m)})^r$, which implies that 
	$$\dfrac{\alpha(I^{(r(m+h-1))})}{r(m+h-1)} \geqslant \dfrac{r(h-1) + r\alpha(I^{(m)})}{r(m+h-1)} = \dfrac{\alpha(I^{(m)})+h-1}{m+h-1},$$
and by taking the limit as $r \rightarrow \infty$.
\end{proof}


\subsection{Generic determinantal ideals}

We now show that the Harbourne--Huneke Containment in Question \ref{question demailly containment general h} holds for generic determinantal ideals. We first prove in Theorem \ref{determinantal case thm} the precise containment stated in Question \ref{question demailly containment general h} to better illustrate our techniques. A more general containment, as in Theorem \ref{thm.StarConf}, will then be shown to hold in Remark \ref{rmk.HHgeneralDI}.

\begin{theorem}[Generic determinantal ideals]\label{determinantal case thm}
Let $\kk$ be a filed. For fixed positive integers $t \leqslant \min\{p,q\}$, let $X$ be a $p \times q$ matrix of indeterminates, let $R = \kk[X]$, and let $I = I_t(X)$ denote the ideal of $t$-minors of $X$. Let $h = (p-t+1)(q-t+1)$ be the height of $I$ in $R$. For all $m,r, c \geqslant 1$, we have
	For any $m, r, c \geqslant 1$, we have
	
	$$I^{(r(h+m-1))} \subseteq \mm^{r(h-1)} \left( I^{(m)}\right)^r.$$
In particular, Demailly-like bound holds for generic determinantal ideals.
\end{theorem}

\begin{proof} 
After possibly replacing $X$ with its transpose, we may assume that $p \geqslant q$. We will use the explicit description of the symbolic powers of the ideals of minors of generic determinantal matrices by Eisenbud, DeConcini, and Procesi \cite{DEPdeterminantal}. We point the reader to \cite[10.4]{determinantalrings} for more details on the subject.
	
	Given a product $\delta = \delta_1 \cdots \delta_u$, where $\delta_i$ is an $s_i$-minor of $X$, $\delta \in I^{(k)}$ if and only if 
	$$\sum^{u}_{i=1} \max \lbrace 0, s_i - t + 1 \rbrace \geqslant k,$$
	and $I^{(k)}$ is generated by such products. Fix such a product $\delta = \delta_1 \cdots \delta_u \in I^{(r(h+m-1))}$, and set $s \colonequals s_1 + \cdots + s_u$. Note that by the rule above, whether or not $\delta$ is in a particular symbolic power of $I$ is unchanged if some $s_i < t$, so we can assume without loss of generality that all $s_i \geqslant t$. We thus have
	$$\sum_{i=1}^u (s_i-t+1) \geqslant r(h+m-1).$$
	We will explicitly write $\delta$ as an element in $( I^{(m)} )^r$. First, note that to write $\delta$ as an element in $I^{(m)}$, it is enough to find a subset of $\lbrace \delta_1, \ldots, \delta_u \rbrace$, say $\delta_1, \ldots, \delta_v$, such that
	$$\sum_{i=1}^v (s_i-t+1) \geqslant m.$$
	If we chose $\delta_1, \ldots, \delta_v$ the best way possible, meaning such that no $\delta_i$ can be deleted, then 
	$$\sum_{i=1}^{v-1} (s_i-t+1) \leqslant m-1,$$ 
	and since $\delta_i$ is a minor of $X$, we must have $s_i \leqslant q$. Thus
	$$\sum_{i=1}^v (s_i-t+1) \leqslant (m-1) + (q-t+1) = m+q-t.$$
	Suppose that 
	$$\sum_{i=1}^v (s_i-t+1) = m+k.$$
	Then $s_v-t +1 \geqslant (m+k) - (m-1) = k+1$, so $s_v \geqslant t+k$. By using a Laplace expansion, we can rewrite $\delta_v$ as a linear combination of minors of size $s_v-k$ with coefficients in $\mm^k$. On the other hand, the product of $\delta_1 \cdots \delta_{v-1}$ by a minor of size $s_v - k \geqslant t$ is still in $I^{(m)}$, so $\delta_1 \cdots \delta_v \in \mm^k I^{(m)}$ for some $k \leqslant q-t$.
	
	By repeating this process $r$ times, we claim that we can extract $\delta_1, \ldots, \delta_w$ (perhaps after reordering) such that 
	$$\delta_1 \cdots  \delta_w \in \mm^{k_1 + \cdots + k_r}\left( I^{(m)} \right)^r \textrm{ and }  \sum_{i=1}^w (s_i - t + 1) = rm + \sum_{j=1}^r k_j\leqslant r (m+q-t).$$
	This is always possible as long as
	$$\sum_{i=1}^w (s_i-t+1) \leqslant \sum_{i=1}^u (s_i-t+1),$$
	which follows from
	$$q-t+1 \leqslant h \implies rm + \sum_{j=1}^r k_j\leqslant r (m+q-t) \leqslant r(h+m-1).$$
	
	As a consequence, the remaining factors in $\delta$ satisfy
	$$\sum_{i=w+1}^u (s_i - t + 1) = \sum_{i=w+1}^u (s_i - t + 1) - r m - \sum_{j=1}^r k_j \geqslant r(h+m-1) - r m - \sum_{j=1}^r k_j.$$
	Since each $s_i$-minor $\delta_i$ has degree $s_i$, and
	$$\sum_{i=w+1}^u s_i \geqslant \sum_{i=w+1}^u (s_i - t + 1) \geqslant r(h-1) - \sum_{j=1}^r k_j,$$
	we conclude that $\delta_{w+1} \cdots \delta_u \in \mm^N$, where $N \geqslant r(h-1) - (k_1 + \cdots + k_r)$.
	Therefore,
	$$\delta = \delta_{1} \cdots \delta_w \delta_{w+1} \cdots \delta_u \in \mm^{k_1 + \cdots + k_m + N} \left( I^{(m)} \right)^r \subseteq \mm^{r(h-1)} \left( I^{(m)} \right)^r.\qedhere$$
\end{proof}

\begin{remark} \label{rmk.HHgeneralDI}
	Since we showed that star configurations satisfy the stronger containment
	$$I^{(r(m+h-1)-h+c)} \subseteq \mm^{(r-1)(h-1)+c-1}(I^{(m)})^r,$$
	it is natural to ask if so do generic determinantal ideals. We shall prove that this stronger inclusion also holds.
	
	We will use the same notation as in the proof of Theorem \ref{determinantal case thm} and assume that $p \geqslant q$. When $t = q$ or, equivalently, when $I$ is the ideal of maximal minors of $X$, it is well-known that $I^{(n)} = I^n$ for all $n \geqslant 1$ --- see \cite[(2.2)]{HochsterCriteria} for the case when $q = p+1$, while the general case follows immediately from \cite{DEPdeterminantal} --- and there is nothing to show. So we will assume that $q > t$, and in particular, we have $p-t+1 \geqslant q-t+1 \geqslant 2$.
	
	When $r=1$, the asserted statement is that
	$$I^{(m+c-1)} \subseteq \mm^{c-1}I^{(m)}$$
	for all $c, m \geqslant 1$, which follows immediately once one shows that $I^{(m+1)} \subseteq \mm I^{(m)}$ for all $m \geqslant 1$. In characteristic $0$, this is given in \cite[Proposition]{EisenbudMazur}; we claim that $I^{(m+1)} \subseteq \mm I^{(m)}$ holds in any characteristic. Consider a product of $s_i$-minors $\delta = \delta_1 \cdots \delta_u \in I^{(m+1)}$ such that $s_i \geqslant t$. If $s_i = t$ for all $i$, then $u \geqslant m+1$, and so we have $\delta_1 \cdots \delta_{u-1} \in I^m \subseteq I^{(m)}$, which implies that $\delta \in \mm I^{(m)}$. If, otherwise, there exists $i$ such that $s_i \geqslant t+1$ then using Laplace expansion, $\delta_i$ can be rewritten as a linear combination of minors of size $s_i-1 \geqslant t$ with coefficients in $\mm$, so that $\delta = \delta_1 \ldots, \delta_u \in \mm I^{(m)}$.
	
	When $r \geqslant 2$, we claim that we can adapt the proof of Theorem \ref{determinantal case thm} to show 
	$$I^{(r(m+h-1)-h+c)} \subseteq \mm^{(r-1)(h-1)+c-1}(I^{(m)})^r,$$
	but we need a bit more work. First, let us try to follow the same steps: we fix a product $\delta = \delta_1 \cdots \delta_u \in I^{(r(h+m-1)-h+c)}$, and start by writing it as a multiple of a product of minors in $\mm^{k_1 + \cdots + k_r} (I^{(m)})^r$ for some $k_j \leqslant q-t$. We started off with 
	$$\sum_{i=1}^u (s_i-t+1) \geqslant r(h+m-1) - h + c$$
	and we need to collect $r$ subsets of $v$ many $\delta_i$ such that
	$$\sum_{i=1}^v (s_i-t+1) \geqslant m,$$
	but all we can guarantee is that
	$$\sum_{i=1}^w (s_i - t + 1) = rm + \sum_{j=1}^r k_j \leqslant r (m+q-t).$$
	This is only possible if 
	$$r(h+m-1) - h + c - r (m+q-t) \geqslant 0.$$
	We claim this inequality does hold given our assumptions that $r \geqslant 2$ and $p \geqslant q > t$. Indeed, writing $Q \colonequals q-t+1 \geqslant 2$ and $P \colonequals p-t+1 \geqslant 2$, we get
	\begin{align*}
	r(h+m-1) - h + c - r (m+q-t) & = r(h-(q-t+1)) - h + c\\
	& = r(PQ - Q) - PQ + c \\ 
	& = PQ(r-1) - rQ + c \\
	& \geqslant 2Q(r-1) - rQ + c \\
	& = (r-2)Q +c \\
	& \geqslant 0.
	\end{align*}
	As in the proof of Theorem \ref{determinantal case thm}, write $\delta = \delta_{1} \cdots \delta_w \delta_{w+1} \cdots \delta_u$ where $\delta_{1} \cdots \delta_w \in \mm^{k_1 + \cdots + k_r} \left( I^{(m)} \right)^r$; the remaining factors in $\delta$ satisfy
	$$\sum_{i=w+1}^u (s_i - t + 1) = \sum_{i=w+1}^u (s_i - t + 1) - r m - \sum_{j=1}^r k_j \geqslant r(h+m-1) - r m -h +c - \sum_{j=1}^r k_j.$$
	Since each $s_i$ minor $\delta_i$ has degree $s_i$, and
	$$\sum_{i=w+1}^u s_i \geqslant \sum_{i=w+1}^u (s_i - t + 1) \geqslant r(h-1) -h+c - \sum_{j=1}^r k_j = (r-1)(h-1)+c-1 - \sum_{j=1}^r k_j,$$
	we conclude that 
	$$\delta = \delta_{1} \cdots \delta_w \delta_{w+1} \cdots \delta_u \in \mm^{(r-1)(h-1)+c-1} I^{(m)}.\qedhere$$
\end{remark}

\begin{remark} \label{rmk.gensym}
	Theorem \ref{determinantal case thm} also holds for determinantal ideals of \emph{symmetric} matrices. Let $t \leqslant p$ be integers, let $Y$ be a $p \times p$ symmetric matrix of indeterminates, meaning $Y_{ij} = Y_{ji}$ for all $1 \leqslant i,j \leqslant p$, and let $I = I_t(Y)$ be the ideal of $t$-minors of $Y$ in $R = \kk[Y]$. By \cite[Proposition 4.3]{MultClassicalVar}, given $s_i$-minors $\delta_i$ of $Y$, the product $\delta = \delta_1 \cdots \delta_u$ is in $I^{(k)}$ if and only if $\sum^{u}_{i=1} \max \lbrace 0, s_i - t + 1 \rbrace \geqslant k$, and $I^{(k)}$ is generated by such elements. The element $\delta = \delta_1 \cdots \delta_u$ has once again degree $s_1 + \cdots + s_u$, and $I$ is a prime of height $h = {p - t + 2 \choose 2} = \frac{(p-t+1)(p-t+2)}{2}$.
	
	Fix $r, m \geqslant 1$. As we did in the case of generic matrices, we start with a product of $s_i$-minors $\delta_i$, say $\delta = \delta_1 \cdots \delta_u \in I^{(r(h+m-1))}$. We again pick aside some of those terms, say $\delta_1, \ldots, \delta_w$, in such a way that we can guarantee $\delta_1 \cdots \delta_w \in \mm^{k_1 + \cdots + k_r} \left( I^{(m)} \right)^r$, where each $k_j \leqslant p-t$. For this to be possible, we again need to check
	$$\sum_{i=1}^w (s_i-t+1) \leqslant \sum_{i=1}^u (s_i-t+1),$$
	which reduces to the same inequality as in Theorem \ref{determinantal case thm}, now with $p=q$:
		$$p-t \leqslant h \implies rm + \sum_{j=1}^r k_j\leqslant r (m+p-t) \leqslant r(h+m-1).$$
	The rest of the proof is exactly the same.\\
	
	Furthermore, by adapting the same argument as in Remark \ref{rmk.HHgeneralDI}, we can show that the containment
	$$I^{(r(m+h-1)-h+c)} \subseteq \mm^{(r-1)(h-1)+c-1}(I^{(m)})^r$$
	also holds. When $r=1$, the proof is exactly the same as in Remark \ref{rmk.HHgeneralDI}. When $r \geqslant 2$, we start with $\delta = \delta_1 \cdots \delta_u \in I^{(r(h+m-1)-h+c)}$, and need to guarantee that it is possible to pick $\delta_1, \ldots, \delta_w$ such that $\delta_1 \cdots \delta_w \in \mm^{k_1 + \cdots + k_m} \left( I^{(m)} \right)^r$, where each $k_j \leqslant p-t$. As before, we need to check that 
	$$\sum_{i=1}^w (s_i-t+1) \leqslant \sum_{i=1}^u (s_i-t+1),$$
	which reduces to checking that
	$$r (m+p-t) \leqslant r(h+m-1)-h+c.$$
	Writting $P \colonequals p-t+1$, we have
	$$r(h+m-1) - h + c - r (m+p-t) = r\left(\dfrac{P(P+1)}{2} - P\right) - \dfrac{P(P+1)}{2} + c .$$
	If $t=p$, $I$ is a principal ideal and the containment holds trivially. When $P \geqslant 2$, we have 
	$$r\left(\dfrac{P(P+1)}{2} - P\right) - \dfrac{P(P+1)}{2} = \dfrac{P}{2}(r(P-1)-(P+1)) = \dfrac{P}{2}((r-1)(P-1)-2) \geqslant 0,$$
	unless $r=2$ and $P=2$. The rest of the argument follows analogously to Remark \ref{rmk.HHgeneralDI}.

	It remains to check the claim in the case $r=2$ and $P=2$. Note that we now have $h=2$. Given $\delta = \delta_1 \cdots \delta_u \in I^{(2m+c)}$, we want to show show that $\delta \in \mm^{c} \left( I^{(m)} \right)^2$. As in the case of generic matrices, we can assume that each $s_i\geqslant t$, and since $p=t+1$, each $\delta_i$ is either a $t$-minor or a $(t+1)$-minor. Let $a$ be the number of $(t+1)$-minors and $b$ is the number of $t$-minors (so $a+b=u$). Since $\sum_{i=1}^u (s_i-t+1) \geqslant 2m+c$, we have $2a+b\geqslant 2m+c$. 
	
	If $m$ is even or $b \geqslant 2$, we can choose $a_1 + a_2 \leqslant a$ and $b_1 + b_2 \leqslant b$ such that $2a_1 + b_1 = m = 2a_2 + b_2$, so that $\delta$ can be written as a product of an element in $\left( I^{(m)} \right)^2$ and an element of degree 
	$$(t+1)(a-a_1-a_2) + t(b-b_1-b_2) \geqslant 2(a-a_1-a_2) + (b-b_1-b_2) \geqslant c.$$
	Therefore, $\delta \in \mm^{c} \left( I^{(m)} \right)^2$. When $b \leqslant 1$, we can reduce to the case $b \geqslant 2$ by viewing either one or two of our $(t+1)$ minors as a linear combination of $t$-minors with coefficients in $\m$.
\end{remark}

\begin{remark} \label{rmk.genskewsym}
	Theorem \ref{determinantal case thm} furthermore holds for pfaffian ideals of \emph{skew} symmetric matrices. Let $t \leqslant \frac{p}{2}$ be integers, let $Z$ be a $p \times p$ skew symmetric matrix of indeterminates, meaning $Z_{ij} = - Z_{ji}$ for $1 \leqslant i < j \leqslant p$ and $Z_{ii} = 0$ for all $i$, and let $I = P_{2t}(Z)$ be the $2t$-pfaffian ideal of $Z$. That is, $I$ is generated by the square roots of the $2t$ principal minors of $Z$. In this case, the symbolic powers of $I$ are generated by products of pfaffians; given $2 s_i$-pfaffians $\delta_i$, \cite[Theorem 2.1]{pfaffiansDN} and \cite[Proposition 4.5]{MultClassicalVar} tell us that $\delta_1 \cdots \delta_u \in I^{(k)}$ if and only if $\sum^{u}_{i=1} \max \lbrace 0, s_i - t + 1 \rbrace \geqslant k$ and $s_i \leqslant \left\lfloor \frac{p}{2} \right\rfloor$. Note also that a $2s_i$-pfaffian $\delta_i$ has degree $s_i$, and by \cite[Corollary 2.5]{pfaffians} the height of $I$ is now $h = {p-2t+2 \choose 2} = \frac{(p-2t+1)(p-2t+2)}{2}$.
	
	In the previous cases, we repeatedly used the fact that any $(t+k)$-minor is a linear combination of $t$-minors with coefficients in $\mm^k$, which is a direct consequence of Laplace expansion. In the case of pfaffians, \cite[Lemma 1.1]{pfaffians} says any skew symmetric matrix $M$ has a cofactor expansion of the form
	$$\textrm{pf } M = \sum_{j \geqslant 2} (-1)^j m_{1j} \textrm{pf } M_{1j}$$
	where $\textrm{pf } A$ denotes the pfaffian of $A$, meaning the square root of $\det(A)$, and $M_{1j}$ denotes the submatrix of $M$ obtained by deleting the first and $j$th rows and columns. As a consequence, any $2(t+k)$-pfaffian of our skew symmetric matrix $Z$ can be written as a linear combination of $2t$-pfaffians with coefficients in $\mm^k$.
	
	Fix $r, m \geqslant 1$. This time, we start with a product $\delta = \delta_1 \cdots \delta_u \in I^{(r(h+m-1))}$ of $2s_i$-pfaffians $\delta_i$, which must then satisfy
	$$\sum_{i=1}^u (s_i - t + 1) \geqslant r(h+m-1).$$
	As before, we start by collecting some of our pfaffians, say $\delta_1, \ldots, \delta_v$, in such a way that
	$$m + k_1 \colonequals \sum_{i=1}^v (s_i - t + 1) \geqslant m.$$
	We can assume 
	$$\sum_{i=1}^{v-1} (s_i - t + 1) \leqslant m-1 \implies s_v - t + 1 \geqslant k_1 +1 \implies k_1 \leqslant s_v - t \leqslant \frac{p-2t}{2}.$$
	Using the trick we described above, we can rewrite $s_v$ as a linear combination of $2t$-pfaffians with coefficients in $\mm^{k_1}$. Therefore, $\delta_1 \cdots \delta_v \in \mm^{k_1} I^{(m)}$. Repeating this idea $r$ many times, we can show that $\delta = \delta_1 \cdots \delta_w \cdots \delta_u$, where $\delta_1 \cdots \delta_w \in \mm^{k_1 + \cdots + k_r} \left( I ^{(m)} \right)^r$, as long as
	$$rm + \sum_{j=1}^r k_j \leqslant r(h+m-1).$$
	Each $k_j \leqslant \frac{p-2t}{2}$ and $h = \frac{(p-2t+1)(p-2t+2)}{2}$, so we need to check that
	$$rm + \frac{r(p-2t)}{2} \leqslant \frac{r(p-2t+1)(p-2t+2)}{2} + r(m-1),$$
	or equivalently,
	$$\frac{p-2t+2}{2} = 1 + \frac{p-2t}{2} \leqslant \frac{(p-2t+1)(p-2t+2)}{2}.$$
	And indeed, this holds, since $p-2t+1 \geqslant 1$. So we have shown that $\delta_1 \cdots \delta_w \in \mm^{k_1 + \cdots + k_r} \left( I ^{(m)} \right)^r$, and the remaining factors $\delta_{w+1} \cdots \delta_u$ satisfy
	$$\sum_{i=w+1}^u (s_i - t+ 1) \geqslant r(h+m-1) - rm - \sum_{j=1}^r k_j.$$
	Just as in the proof of Theorem \ref{determinantal case thm}, the degree of $\delta_{w+1} \cdots \delta_u$ is 
	$$\sum_{i=w+1}^u s_i \geqslant \sum_{i=w+1}^u (s_i - t+ 1) \geqslant r(h-1) - \sum_{j=1}^r k_j,$$
	so that $\delta \in \mm^{k_1 + \cdots k_r+N} \left(I^{(m)}\right)^r$ for $N \geqslant r(h-1) - (k_1 + \cdots + k_r)$. Therefore, $\delta \in \mm^{r(h-1)} \left( I^{(m)} \right)$.
	
	\
	
	Finally, we can once more show the stronger containment 
	$$I^{(r(m+h-1)-h+c)} \subseteq \mm^{(r-1)(h-1)+c-1}(I^{(m)})^r.$$
	As in the previous cases, the case $r=1$ follows from the Laplace expansion trick. When $r \geqslant 2$, we start with a product $\delta = \delta_1 \cdots \delta_u$ of $2s_i$-pfaffians $\delta_i$ satisfying
	$$\sum_{i=1}^u (s_i - t + 1) \geqslant r(m+h-1)-h+c$$
	and collect a subset $\delta_1 \cdots \delta_w \in \mm^{k_1 + \cdots + k_r} \left( I^{(m)} \right)^r$ with each $k_j \leqslant \frac{p-2t}{2}$. We can only do this if 
	$$r(m+h-1)-h+c \geqslant rm + \sum_{j=1}^r k_j,$$
	so it is sufficient to show
	$$(r-1)h +rm - r + c = r(m+h-1)-h+c \geqslant rm + \frac{r(p-2t)}{2}.$$
	Setting $P \colonequals p-2t$ so that $2h = (P+1)(P+2)$, we want to show that
	$$(r-1)(P+1)(P+2) - 2r + 2c - rP \geqslant 0,$$
	or equivalently,
	$$rP(P+2) + 2c \geqslant (P+1)(P+2).$$
	Whenever $r \geqslant 2$, we also have $rP \geqslant P+1$, and our desired inequality holds. So we have shown $\delta_1 \cdots \delta_w \in \mm^{k_1 + \cdots + k_r} \left( I^{(m)} \right)^r$, and $\delta_{w+1} \cdots \delta_u$ has degree
	$$N = \sum_{i=w+1}^u s_i \geqslant \sum_{i=w+1}^u (s_i-t+1) \geqslant r(m+h-1)-h+c -rm - \sum_{j=1}^r k_j,$$
	so $\delta \in \mm^{k_1 + \cdots + k_r + N} \left( I^{(m)} \right)^r$, where
	$$k_1 + \cdots + k_r + N \geqslant r(m+h-1)-h+c -rm = r(h-1)-h+c = (r-1)(h-1)+c-1,$$
	as desired.
\end{remark}

\bibliographystyle{alpha}
\bibliography{References}

\end{document}